\documentclass[11pt]{article}

\usepackage[a4paper,margin=1in]{geometry}
\usepackage{amsmath,amssymb,amsthm,mathtools,lineno,enumitem}
\usepackage{xcolor}
\usepackage{aliascnt}

\usepackage[colorlinks=true,
linkcolor=blue!50!black,
citecolor=blue!50!black,
urlcolor=blue!50!black]{hyperref}
\usepackage[nameinlink,noabbrev,capitalise]{cleveref}

\theoremstyle{plain}

\newtheorem{theorem}{Theorem}[section]
\crefname{theorem}{Theorem}{Theorems}
\Crefname{theorem}{Theorem}{Theorems}

\newaliascnt{proposition}{theorem}
\newtheorem{proposition}[proposition]{Proposition}
\aliascntresetthe{proposition}
\crefname{proposition}{Proposition}{Propositions}
\Crefname{proposition}{Proposition}{Propositions}

\newaliascnt{corollary}{theorem}

\aliascntresetthe{corollary}
\crefname{corollary}{Corollary}{Corollaries}
\Crefname{corollary}{Corollary}{Corollaries}

\newaliascnt{lemma}{theorem}
\newtheorem{lemma}[lemma]{Lemma}
\aliascntresetthe{lemma}
\crefname{lemma}{Lemma}{Lemmas}
\Crefname{lemma}{Lemma}{Lemmas}

\newaliascnt{conjecture}{theorem}
\newtheorem{conjecture}[conjecture]{Conjecture}
\aliascntresetthe{conjecture}
\crefname{conjecture}{Conjecture}{Conjectures}
\Crefname{conjecture}{Conjecture}{Conjectures}

\newaliascnt{problem}{theorem}

\aliascntresetthe{problem}
\crefname{problem}{Problem}{Problems}
\Crefname{problem}{Problem}{Problems}

\newaliascnt{claim}{theorem}

\aliascntresetthe{claim}
\crefname{claim}{Claim}{Claims}
\Crefname{claim}{Claim}{Claims}

\newaliascnt{observation}{theorem}

\aliascntresetthe{observation}
\crefname{observation}{Observation}{Observations}
\Crefname{observation}{Observation}{Observations}

\newaliascnt{setup}{theorem}

\aliascntresetthe{setup}
\crefname{setup}{Setup}{Setups}
\Crefname{setup}{Setup}{Setups}

\newaliascnt{myth}{theorem}

\aliascntresetthe{myth}
\crefname{myth}{Myth}{Myths}
\Crefname{myth}{Myth}{Myths}

\newaliascnt{fact}{theorem}

\aliascntresetthe{fact}
\crefname{fact}{Fact}{Facts}
\Crefname{fact}{Fact}{Facts}

\newaliascnt{algorithm}{theorem}

\aliascntresetthe{algorithm}
\crefname{algorithm}{Algorithm}{Algorithms}
\Crefname{algorithm}{Algorithm}{Algorithms}

\theoremstyle{remark}

\newaliascnt{remark}{theorem}

\aliascntresetthe{remark}
\crefname{remark}{Remark}{Remarks}
\Crefname{remark}{Remark}{Remarks}

\newaliascnt{example}{theorem}

\aliascntresetthe{example}
\crefname{example}{Example}{Examples}
\Crefname{example}{Example}{Examples}

\theoremstyle{definition}

\newaliascnt{definition}{theorem}

\aliascntresetthe{definition}
\crefname{definition}{Definition}{Definitions}
\Crefname{definition}{Definition}{Definitions}

\newaliascnt{construction}{theorem}

\aliascntresetthe{construction}
\crefname{construction}{Construction}{Constructions}
\Crefname{construction}{Construction}{Constructions}

\newaliascnt{question}{theorem}
\newtheorem{question}[question]{Question}
\aliascntresetthe{question}
\crefname{question}{Question}{Questions}
\Crefname{question}{Question}{Questions}

\numberwithin{equation}{section}
\crefname{equation}{Equation}{Equations}
\Crefname{equation}{Equation}{Equations}

\crefformat{section}{\S#2#1#3}
\Crefformat{section}{Section~#2#1#3}
\crefformat{subsection}{\S#2#1#3}
\Crefformat{subsection}{Subsection~#2#1#3}
\crefformat{subsubsection}{\S#2#1#3}
\Crefformat{subsubsection}{Subsubsection~#2#1#3}
\crefrangeformat{section}{\S\S#3#1#4--#5#2#6}
\Crefrangeformat{section}{Sections~#3#1#4--#5#2#6}
\crefmultiformat{section}
  {\S\S#2#1#3}
  { and~#2#1#3}
  {, #2#1#3}
  {, and~#2#1#3}
\Crefmultiformat{section}
  {Sections~#2#1#3}
  { and~#2#1#3}
  {, #2#1#3}
  {, and~#2#1#3}

\newcommand{\ind}{\operatorname{ind}}
\newcommand{\pc}{\operatorname{pc}}
\newcommand{\E}{\mathbb E}
\newcommand{\Prob}{\mathbb P}

\newcommand{\Ffac}{\mathcal F}

\newcommand{\pan}{\mathrm{pan}}

\begin{document}

\title{Pancyclicity of graphs perturbed by a random $F$-factor}
\author{
Dingjia Mao
\thanks{Department of Mathematics, University of California, Irvine.
{Email: \tt{dingjiam@uci.edu}.}}
\and 
Feihong Yuan
\thanks{School of Mathematics and Statistics, Xi'an Jiaotong University.
Email:  {\tt  fhyuan1@gmail.com.}}
\and 
Wenling Zhou
\thanks{School of Mathematics, Shandong University, Jinan, China.
Email: {\tt gracezhou@sdu.edu.cn.}}
}
\date{}
\maketitle
\begin{abstract}
We determine the sharp minimum-degree threshold for Hamiltonicity in graphs
perturbed by a uniformly random $K_r$-factor, resolving a conjecture of
Espuny D\'iaz and Gir\~ao [\emph{Random Structures Algorithms}, 2023].  In
fact, we prove the stronger pancyclic statement.  Let $\alpha^*(K_r)$ and
$\alpha_{\pan}^*(K_r)$ denote the Hamiltonicity and pancyclicity thresholds, respectively. We show that $\alpha^*(K_r)=\alpha_{\pan}^*(K_r)=\rho_r$, where $\rho_r$ is the unique
positive solution of $x^r+rx-1=0$.  The proof is obtained from a general
framework for perturbations by a uniformly random $F$-factor, where $F$ is an
arbitrary fixed connected graph.
\end{abstract}

\section{Introduction}
Hamiltonicity is one of the most fundamental and extensively studied notions in extremal and probabilistic graph theory. 
A \emph{Hamilton cycle} in a graph $G$ is a cycle that contains all vertices of $G$, and $G$ is \emph{Hamiltonian} if it contains a Hamilton cycle. Since deciding whether a graph contains a Hamilton cycle is NP-complete, as
shown by Karp~\cite{Karp1972}, it is natural to seek sufficient conditions for the existence of a Hamilton cycle. The classical theorem of Dirac~\cite{Dirac52}
states that every graph on $n\ge3$ vertices with minimum degree at least $n/2$
is Hamiltonian. 
In the random setting, the binomial random graph $G(n,p)$
becomes Hamiltonian at the point at which the last vertices of degree less than
two disappear; more precisely, Hamiltonicity in
$G(n,p)$ has threshold $(\log n+\log\log n)/n$, as follows from the work of
P\'osa~\cite{Posa76}, Kor\v{s}unov~\cite{Korsunov77}, and Koml\'os and
Szemer\'edi~\cite{KS83}.

The model of \emph{randomly perturbed graphs}, introduced by Bohman, Frieze and Martin~\cite{BFM03}, interpolates between these two viewpoints. In this model, one starts with a deterministic graph $H$ and adds a random graph on the same vertex set. Bohman, Frieze and Martin proved that, for every fixed
$\alpha>0$, there exists $C=C(\alpha)>0$ such that, if $|V(H)|=n$ and
$\delta(H)\ge \alpha n$, then $H\cup G(n,C/n)$ is Hamiltonian with high
probability.  Here and throughout, an event depending on $n$ holds \emph{with
high probability}, abbreviated by \emph{w.h.p.}, if its probability tends to
one as $n\to\infty$.  This result initiated a broad line of work on spanning
structures in randomly perturbed graphs, including bounded-degree spanning
trees~\cite{BHKMPP19,JK20,KKS17}, general bounded-degree spanning
graphs~\cite{BMPP20}, powers of Hamilton cycles~\cite{ADRRS21,ADR23,BPSS24,DRRS20},
Ramsey properties~\cite{DT20}, and tilings and factors~\cite{AKRT26,BTW19,HMT21}.
Hypergraph analogues have also been developed, for instance for Hamilton cycles
and factors in randomly perturbed hypergraphs~\cite{yulin-2022,HZ20,MM18}.

Most results in this direction add independently random edges. A rather different
perturbation is obtained by adding a sparse random regular graph, or more
generally a random factor, where the random edges are highly dependent.  We
write $G_{n,d}$ for a graph chosen uniformly at random from the set of all
$d$-regular graphs on $n$ vertices, whenever this set is non-empty.  It is known
that $G_{n,d}$ is Hamiltonian w.h.p. for every fixed $d\ge3$; see
\cite{CFR02,KSVW01,RW92,RW94} and the survey of Wormald~\cite{Wor99}. 
Thus,
for random regular perturbations, the genuinely sparse cases are $d=1$ and
$d=2$. Espuny D\'iaz and Gir\~ao~\cite{DG23} initiated the study of
Hamiltonicity and pancyclicity in graphs perturbed by $G_{n,d}$ for
$d\in\{1,2\}$. For $d=2$, they obtained a result similar to those obtained
in the binomial setting: 
if $\delta(H)\ge \omega((\log n/n)^{1/4})n$, then w.h.p.~$H \cup G_{n, 2}$ is \emph{pancyclic}, that is, $H \cup G_{n, 2}$ contains a cycle of every length
\(3,4,\ldots,n\). More recently, Dragani\'c and Keevash~\cite{draganic2025p} determined the
corresponding Hamiltonicity threshold asymptotically, showing that it is
$(1+o(1))\sqrt{n\log n/2}$.

The case $d=1$, that is, perturbation by a uniformly random perfect matching,
has a different behaviour.  Espuny D\'iaz and Gir\~ao~\cite{DG23} proved that, for every $\varepsilon>0$, if $\delta(H)\ge
(1+\varepsilon)(\sqrt2-1)n$, then $H\cup G_{n,1}$ is pancyclic w.h.p.; they
also showed that the constant $\sqrt2-1$ is best possible for Hamiltonicity. 
Since a perfect matching is a $K_2$-factor, it is natural to ask what happens
when the perturbation is a uniformly random $F$-factor for a fixed graph $F$.

Throughout the paper, $F$ is a fixed connected graph on $m\ge2$ vertices, with
$V(F)=\{1,\ldots,m\}$, and we assume that $m$ divides $n$.  A uniformly random
$F$-factor $\Ffac$ on an $n$-vertex set $V$ is generated by taking a uniformly
random ordering $v_1,\ldots,v_n$ of $V$ and, for each
$i=0,\ldots,n/m-1$, placing a copy of $F$ on
$\{v_{im+1},\ldots,v_{im+m}\}$ by identifying $j\in V(F)$ with $v_{im+j}$.

For a fixed connected graph $F$,
the corresponding \emph{minimum-degree threshold of Hamiltonicity} is defined as
\[
\begin{aligned}
\alpha^*(F)=\inf \bigl\{ a\in[0,1]:\;&
\text{for every } \varepsilon>0,\text{ there exists } n_0\in\mathbb N
\text{ such that for every graph } \\
&H \text{ on } n\ge n_0 \text{ vertices }
\text{with } \delta(H)\ge (a+\varepsilon)n
\text{ and } |V(F)|\mid n,\\
&\text{ w.h.p. } H\cup \mathcal F \text{ is Hamiltonian}
\bigr\}.
\end{aligned}
\]
We define the pancyclicity threshold
$\alpha_{\pan}^*(F)$ analogously, with ``Hamiltonian'' replaced by
``pancyclic''.  Clearly $\alpha^*(F)\le\alpha_{\pan}^*(F)$.

In \cite{DG23}, the authors showed that $\alpha_{\pan}^*(K_2)=\sqrt{2}-1$, and proposed the following conjecture.
\begin{conjecture}[{\cite[Conjecture 14]{DG23}}]\label{conjrsa}
For all $r \ge 2$, we have that $\alpha^*(K_r)$ is the unique real positive solution to the equation $x^r+rx-1=0$.
\end{conjecture}

Our main result confirms this conjecture and strengthens it from Hamiltonicity
to pancyclicity.  Let $\rho_r$ denote the unique positive solution of
$x^r+rx-1=0$.

\begin{theorem}\label{thm:clique}
For every integer $r\ge 2$, we have
\[
        \alpha^*(K_r)=\alpha_{\pan}^*(K_r)=\rho_r .
\]
\end{theorem}

The proof proceeds through a general framework for arbitrary connected $F$.
For a graph $G$, let $\ind(G)$ denote its independence number, and let $\pc(G)$
denote the minimum number of vertex-disjoint paths covering $V(G)$, where
isolated vertices are allowed as paths.  We use the conventions
$\ind(\emptyset)=\pc(\emptyset)=0$.   

For $x\in[0,1]$, let $S_x$ be the random subset of $V(F)$ obtained by including each vertex independently with probability $x$, and define
\begin{equation}\label{eq:exp-olynomials}
\Psi_F(x):=\E\,\pc(F-S_x),
\qquad
\Phi_F(x):=\E\,\ind(F-S_x),    
\end{equation}
where $F-S_x$ denotes the induced subgraph on $V(F)\setminus S_x$. 
As shown in \cref{prop:parameters}, the equations $\Psi_F(x)=mx$ and
$\Phi_F(x)=mx$ have unique solutions in $(0,1)$; we denote them by
$\tau_{\pc}(F)$ and $\tau_{\ind}(F)$, respectively.  Our general theorem is as
follows.

\begin{theorem}\label{thm:general}
Let $F$ be a connected graph on $m\ge2$ vertices.  Then
\[
        \tau_{\pc}(F)
        \le \alpha^*(F)
        \le \alpha_{\pan}^*(F)
        \le \tau_{\ind}(F).
\]
\end{theorem}

We first explain why the two functions in \eqref{eq:exp-olynomials} arise.
Suppose first that $H=K_{A,B}$ is the complete bipartite graph on $n$ vertices, where $|A|=(x+o(1))n$ and $B$ is the larger
part. If $H\cup\Ffac$ contains a Hamilton cycle, then deleting $A$ from that
cycle leaves at most $|A|$ vertex-disjoint paths covering $B$. Since $H$ has
no edges inside $B$, all these paths lie in $\Ffac[B]$, and hence
$\pc(\Ffac[B])\le |A|$. On the other hand, the concentration estimate in
\cref{lem:factor-set-concentration} gives
$\pc(\Ffac[B])=(\Psi_F(x)/m+o(1))n$. Thus the complete bipartite host is an
obstruction whenever $\Psi_F(x)>mx$, which accounts for the lower threshold
$\tau_{\pc}(F)$.

For the upper bound, assume that $\delta(H)\ge xn$ and let $I$ be an
independent set in $H\cup\Ffac$. Choosing $v\in I$ and writing
$C_v:=V(H)\setminus N_H(v)$, we have $I\subseteq C_v$ and
$|C_v|\le(1-x)n$. Uniform concentration over the sets $C_v$, together
with the monotonicity of $\Phi_F$ (proved in~\cref{prop:parameters}), yields
$|I|\le(\Phi_F(x)/m+o(1))n$. Thus, when $x>\tau_{\ind}(F)$, the independence number of
$H\cup\Ffac$ lies below its minimum degree by a linear margin.  A separate
joining estimate is then used to establish the vertex-connectivity required
by a Chv\'atal--Erd\H{o}s type Hamiltonicity result.

We conclude the introduction by briefly comparing our proof with the argument of Espuny D\'iaz and
Gir\~ao~\cite{DG23}. For the lower bound, we use the same unbalanced
complete bipartite obstruction, but the number of random matching
edges in the larger part is replaced by the path-cover statistic $\Psi_F$,
which applies to an arbitrary fixed connected
graph $F$. 
The upper-bound argument is structurally different from that
of~\cite{DG23}, whose proofs for $G_{n,1}$ and $G_{n,2}$ use an
absorbing path and iterative merging of a path--cycle system.  Our proof
does not construct an absorber.  Instead, concentration of $\Phi_F$ and a
joining estimate control, respectively, the independence number and the
connectivity of $H\cup\Ffac$, allowing us to apply a Chv\'atal--Erd\H{o}s result.  Short
cycles are obtained by closing deterministic paths with edges of the random
factor, and long cycles by applying the same Hamiltonicity argument to
suitable initial segments.  
This approach applies to every fixed
connected graph $F$.

The rest of the paper is organised as follows.  In \cref{sec:parameters} we
establish the elementary properties of $\Psi_F$ and $\Phi_F$ and deduce
\cref{thm:clique} from \cref{thm:general}.  In \cref{sec:random-factors} we
prove the probabilistic estimates for random $F$-factors used later.  The lower
bound in \cref{thm:general} is proved in \cref{sec:lower}.  The upper bound,
including pancyclicity, is proved in \cref{sec:upper}.  We conclude with
remarks on non-complete factors in \cref{sec:concluding}.

\section{The parameters}\label{sec:parameters}

In this section, we record the basic properties of the two polynomials
$\Psi_F$ and $\Phi_F$.  Throughout the section, $F$ is a connected graph on
$m\ge2$ vertices, and we write $V:=V(F)$.
We shall use the elementary inequality
\begin{equation}\label{eq:pc-le-ind}
        \pc(G)\le \ind(G)
\end{equation}
for every graph $G$.  Indeed, let $P_1,\ldots,P_q$ be a path cover of $G$ of
minimum size, so that $q=\pc(G)$.  Choose one endpoint $p_i$ from each path
$P_i$.  If $p_ip_j\in E(G)$ for some $i\ne j$, then, after reversing one of the
two paths if necessary, $P_i$ and $P_j$ can be joined through the edge
$p_ip_j$.  This gives a path cover with $q-1$ paths, a contradiction.   Therefore \(\{p_1,\ldots,p_q\}\) is independent, and
\eqref{eq:pc-le-ind} follows.

For every $x\in[0,1]$, applying \eqref{eq:pc-le-ind} to each induced subgraph \(F-S_x\) gives
\begin{equation}\label{eq:PsiPhi}
\Psi_F(x)\le \Phi_F(x).
\end{equation}  
Moreover, by the definition in~\eqref{eq:exp-olynomials},
we have 
\[
\Psi_F(x)
=
\sum_{S\subseteq V(F)}
\pc(F-S)x^{|S|}(1-x)^{m-|S|},
\quad
\Phi_F(x)
=
\sum_{S\subseteq V(F)}
\ind(F-S)x^{|S|}(1-x)^{m-|S|}.
\]

The next proposition records the monotonicity properties of
\(\Psi_F\) and \(\Phi_F\).  These properties ensure that the two threshold 
parameters $\tau_{\pc}(F)$ and $\tau_{\ind}(F)$ are well-defined.

\begin{proposition}\label{prop:parameters}
Let $F$ be a connected graph on $m\ge2$ vertices. Then the following hold.
\begin{enumerate}[label=\textup{(\roman*)}]
\item $\Phi_F(x)$ is non-increasing on $[0,1]$.
\item The function $\Delta_F(x):=mx-\Psi_F(x)$ is strictly increasing on \([0,1]\).
\item The equations $\Psi_F(x)=mx$  and $\Phi_F(x)=mx$ each have a unique solution in $(0,1)$. Moreover $\tau_{\pc}(F)\le \tau_{\ind}(F)$.
\end{enumerate}
\end{proposition}

\begin{proof}
We begin with a derivative identity. For any function \(h:2^V\to\mathbb R\),  define
\[
\Theta_h(x)
:=
\sum_{A\subseteq V}
h(A)x^{|A|}(1-x)^{m-|A|}.
\]
Since the sum is finite, differentiating term by term gives, for \(0<x<1\),
\[
\begin{aligned}
\Theta_h'(x)
&=
\sum_{A\subseteq V}
h(A)
\Bigl(
|A|x^{|A|-1}(1-x)^{m-|A|}
-
(m-|A|)x^{|A|}(1-x)^{m-|A|-1}
\Bigr)                                                   \\
&=
\sum_{A\subseteq V}
h(A)
\Bigl(
\sum_{u\in A}x^{|A|-1}(1-x)^{m-|A|}
-
\sum_{u\in V\setminus A}x^{|A|}(1-x)^{m-|A|-1}
\Bigr).
\end{aligned}
\]
We now rewrite the two double sums with the same index set.  For the positive
part, the summation is over all pairs \((A,u)\) with \(u\in A\).  Setting $T=A\setminus\{u\}$,
we obtain \(T\subseteq V\setminus\{u\}\), \(A=T\cup\{u\}\), and thus
\[
\sum_{A\subseteq V}\sum_{u\in A}
h(A)x^{|A|-1}(1-x)^{m-|A|}
=
\sum_{u\in V}
\sum_{T\subseteq V\setminus\{u\}}
h(T\cup\{u\})x^{|T|}(1-x)^{m-1-|T|}.
\]
For the negative part, the summation is over all pairs \((A,u)\) with
\(u\notin A\).  Setting \(T=A\), we have \(T\subseteq V\setminus\{u\}\), and
therefore
\[
\sum_{A\subseteq V}\sum_{u\in V\setminus A}
h(A)x^{|A|}(1-x)^{m-|A|-1}
=
\sum_{u\in V}
\sum_{T\subseteq V\setminus\{u\}}
h(T)x^{|T|}(1-x)^{m-1-|T|}.
\]
Combining the last two identities yields
\begin{equation}\label{eq:theta-derivative}
\Theta_h'(x)
=
\sum_{u\in V}
\sum_{T\subseteq V\setminus\{u\}}
\bigl(h(T\cup\{u\})-h(T)\bigr)
x^{|T|}(1-x)^{m-1-|T|}.
\end{equation}

We first apply \eqref{eq:theta-derivative} with $h(T):=\ind(F-T)$ for each $T\subseteq V$.
If \(u\notin T\), then \(F-(T\cup\{u\})\) is an induced subgraph of \(F-T\).
Hence $h(T\cup\{u\})\le h(T)$.
All weights in \eqref{eq:theta-derivative} are non-negative for
\(0<x<1\), and therefore \(\Phi_F'(x)\le0\) on \((0,1)\).  Since
\(\Phi_F\) is a polynomial, it follows that \(\Phi_F\) is non-increasing on
\([0,1]\).

Next apply \eqref{eq:theta-derivative} with $h(T):=\pc(F-T)$.
Fix \(T\subseteq V\) and \(u\in V\setminus T\).  Let \(\mathcal P\) be a
minimum path cover of \(F-T\).  Delete \(u\) from the path of \(\mathcal P\)
which contains it.  If \(u\) is an endpoint of that path, the number of paths
does not increase; if \(u\) is an internal vertex, that path is split into two
paths; and if \(u\) is an isolated vertex, the number of paths decreases by
one.  In all cases we obtain a path cover of \(F-(T\cup\{u\})\) with at most
one more path than \(\mathcal P\).  Hence $h(T\cup\{u\})\le h(T)+1$.
Using \eqref{eq:theta-derivative}, we get
\[
\Psi_F'(x)
\le
\sum_{u\in V}
\sum_{T\subseteq V\setminus\{u\}}
x^{|T|}(1-x)^{m-1-|T|}     
=
\sum_{u\in V}1
=m,
\]
where the inner sum is \(1\) by the binomial theorem.

We in fact need the strict form of this inequality.  For each \(u\in V\), take $T=V\setminus\{u\}$.
Then \(F-T\) is the one-vertex graph on \(u\), while
\(F-(T\cup\{u\})=\emptyset\).  Therefore $h(T\cup\{u\})-h(T)=0-1=-1$.
In the preceding upper bound this term was bounded by \(1\), and its weight is
\(x^{m-1}\).  Thus, for every \(0<x<1\),
\[
        \Psi_F'(x)
        \le
        m-2mx^{m-1}
        <
        m.
\]
Consequently,
\[
        \Delta_F'(x)=m-\Psi_F'(x)>0
        \qquad\text{for every }0<x<1.
\]
Since \(\Delta_F\) is continuous, it follows that \(\Delta_F\) is strictly
increasing on \([0,1]\).

It remains to prove the assertions about the two roots.  The function
$f_{\ind}(x):=mx-\Phi_F(x)$ is strictly increasing on $[0,1]$, because
$x\mapsto mx$ is strictly increasing and $\Phi_F$ is non-increasing.  Moreover
$f_{\ind}(0)=-\ind(F)<0$ and $f_{\ind}(1)=m>0$.  Hence $f_{\ind}$ has a unique
zero in $(0,1)$, which is precisely the unique solution of $\Phi_F(x)=mx$.
Similarly, $\Delta_F(0)=-\pc(F)<0$ and $\Delta_F(1)=m>0$, and since
$\Delta_F$ is continuous and strictly increasing, $\Psi_F(x)=mx$ has a unique
solution in $(0,1)$.

Finally, let $t=\tau_{\pc}(F)$.  Then $mt=\Psi_F(t)\le\Phi_F(t)$ by
\eqref{eq:PsiPhi}.  Since $f_{\ind}$ is strictly increasing and its unique zero
is $\tau_{\ind}(F)$, the inequality $f_{\ind}(t)\le0$ implies
$t\le\tau_{\ind}(F)$.  This proves the proposition.
\end{proof}

We now use \cref{thm:general} to prove \cref{thm:clique}.

\begin{proof}[Proof of \cref{thm:clique}]
Let $F=K_r$.  If $S\subsetneq V(F)$, then $F-S$ is a non-empty clique, and
therefore $\ind(F-S)=\pc(F-S)=1$; if $S=V(F)$, both quantities are zero.  Hence
$\Phi_{K_r}(x)=\Psi_{K_r}(x)=1-x^r$.  The common equation
$1-x^r=rx$ is equivalent to $x^r+rx-1=0$.  The polynomial $x^r+rx-1$ is
strictly increasing on $[0,\infty)$, negative at $0$, and positive at $1$, so
it has a unique positive root, denoted by $\rho_r$.  Thus
$\tau_{\pc}(K_r)=\tau_{\ind}(K_r)=\rho_r$.  By \cref{thm:general},
\[
\rho_r\le \alpha^*(K_r)\le\alpha^*_{\pan}(K_r)\le \rho_r,
\]
which proves the theorem.
\end{proof}

\section{Random \texorpdfstring{$F$}{F}-factors}\label{sec:random-factors}

In this section, we shall collect some elementary properties of random $F$-factors. Our main probabilistic tool is the following concentration inequality of McDiarmid (see \cite{McD98}). Throughout this section, for integers \(N\ge0\) and \(r\ge0\), we write
\((N)_r=N(N-1)\cdots(N-r+1)\), with \((N)_0=1\), for the falling factorial.

\begin{lemma}[McDiarmid's inequality for random permutations]\label{lem:mcdiarmid}
Let $\pi$ be a uniformly random permutation of an $n$-element set, and let
$X=X(\pi)$ be a real-valued function of $\pi$. Suppose that there is a
constant $c>0$ such that swapping two entries of $\pi$ changes the value of
$X$ by at most $c$. Then, for every $t>0$,
\[
        \Pr\bigl(|X-\mathbb E X|\ge t\bigr)
        \le
        2\exp\left(-\frac{2t^2}{c^2 n}\right).
\]
\end{lemma}

The next lemma gives the concentration estimate for the sum of the independence
numbers, or of the path-cover numbers, over the components of a uniformly
random $F$-factor.

\begin{lemma}\label{lem:factor-set-concentration}
Let $F$ be a graph on $m$ vertices, and let $h\in\{\ind,\pc\}$.  For
$x\in[0,1]$, define
\[
        \Theta_h(x):=\sum_{S\subseteq V(F)}h(F-S)x^{|S|}(1-x)^{m-|S|}.
\]
Let $n$ be a sufficiently large integer divisible by $m$, and let
$\mathcal F=\{F_1,\dots,F_{n/m}\}$ be a uniformly random $F$-factor on an
$n$-vertex set $V$.  Let $\mathcal U$ be a deterministic family of subsets of
$V$ with $|\mathcal U|\le 2n$.  For $U\subseteq V$, put
$Z_h(U):=\sum_{i=1}^{n/m}h(F_i[U])$, where $F_i[U]$ denotes the subgraph of
$F_i$ induced by $V(F_i)\cap U$.  Then, for every fixed $\xi>0$, with
probability at least $1-\exp(-\Omega_{F,\xi}(n))$, every $U\in\mathcal U$
satisfies
\[
        \left|Z_h(U)-\frac{n}{m}\Theta_h\left(1-\frac{|U|}{n}\right)\right|
        \le \xi n.
\]
\end{lemma}

\begin{proof}
We first fix $U\subseteq V$ and compute the expectation of $Z_h(U)$.
Consider one fixed copy of \(F\), and let
\(\varphi:V(F)\to V\) be its labelled embedding.  The map \(\varphi\)
is uniformly distributed over all injections from \(V(F)\) into \(V\),
of which there are \((n)_m\).  For \(S\subseteq V(F)\), the event
\(\varphi^{-1}(V\setminus U)=S\) means that the vertices in \(S\) are
mapped to \(V\setminus U\), while the vertices in \(V(F)\setminus S\)
are mapped to \(U\).  There are
\((n-|U|)_{|S|}(|U|)_{m-|S|}\) such injections.  Hence this event has
probability
\[
\frac{(n-|U|)_{|S|}(|U|)_{m-|S|}}{(n)_m}.
\]
Therefore, by linearity of expectation,
\[
        \mathbb E Z_h(U)
        =
        \frac{n}{m}
        \sum_{S\subseteq V(F)}
        h(F-S)
        \frac{(n-|U|)_{|S|}(|U|)_{m-|S|}}{(n)_m}.
\]
Since $m$ is fixed, uniformly in $U$ and $S\subseteq V(F)$ we have
\[
        \frac{(n-|U|)_{|S|}(|U|)_{m-|S|}}{(n)_m}
        =
        \left(1-\frac{|U|}{n}\right)^{|S|}\left(\frac{|U|}{n}\right)^{m-|S|}
        +
        O\left(\frac1n\right).
\]
Hence
\begin{equation}\label{expe}
  \mathbb E Z_h(U)
        =
        \frac{n}{m}\Theta_h\left(1-\frac{|U|}{n}\right)+O(1).
\end{equation}

View
$Z_h(U)$ as a function of the random permutation $v_1,\dots,v_n$ used to
generate the $F$-factor.  If two entries of the permutation are swapped,  at most two
copies of $F$ are affected.
Moreover, for every induced subgraph $J$ of $F$, we have
$0\le h(J)\le m$.
Thus swapping two entries changes $Z_h(U)$ by at most $2m$.  Therefore, by
McDiarmid's inequality for random permutations, we obtain
\[
        \Pr\left(
        \left|Z_h(U)-\mathbb E Z_h(U)\right|>\frac{\xi n}{2}
        \right)
        \le
        2\exp\left(-\frac{\xi^2}{8m^2}n\right).
\]

By taking a union bound over all $U\in \mathcal U$, we have
\[
        \Pr\left(
        \exists\, U\in\mathcal U:
        \left|Z_h(U)-\mathbb E Z_h(U)\right|>\frac{\xi n}{2}
        \right)
        \le
        4n\exp\left(-\frac{\xi^2}{8m^2}n\right)  
        \le \exp(-\Omega_{F,\xi}(n)).
\]
Hence,  with probability at least
$1-\exp(-\Omega_{F,\xi}(n))$, we have $ \left|Z_h(U)-\mathbb E Z_h(U)\right|
        \le \xi n/2$
for every $U\in\mathcal U$.
Combining this with \eqref{expe},
we obtain,  for all $U\in\mathcal U$,
\[
        \left|
        Z_h(U)
        -
        \frac{n}{m}\Theta_h\left(1-\frac{|U|}{n}\right)
        \right|
        \le
        \xi n.
\]
This completes the proof.
\end{proof}

The next lemma shows that a uniformly random $F$-factor joins any two disjoint
linear-sized sets with exponentially high probability.

\begin{lemma}\label{lem:linear-pairs-joined}
Let $0<c<1$, and let $n$ be a sufficiently large integer divisible by $m$. Let
$V$ be an $n$-vertex set, and let $X,Y\subseteq V$ be disjoint subsets with
$|X|,|Y|\ge cn$. Let $\mathcal F$ be a uniformly random $F$-factor on $V$, where
$F$ is a non-empty graph on $m$ vertices. Then
\[
        \Pr\bigl(E_{\mathcal F}(X,Y)=\emptyset\bigr)
        \le \exp(-\Omega_{F,c}(n)).
\]
Here $E_{\mathcal F}(X,Y)$ denotes the set of edges of $\mathcal F$ with one
endpoint in $X$ and the other endpoint in $Y$.
\end{lemma}

\begin{proof}
Since $F$ is non-empty, we may fix an edge $ab\in E(F)$.
Let $M$ be the number of copies of $F$ in $\mathcal F$ for which the
vertex corresponding to $a$ lies in $X$ and the vertex corresponding to
$b$ lies in $Y$.  For each copy, the ordered pair of vertices corresponding
to $a$ and $b$ is uniformly distributed over all ordered pairs of distinct
vertices of $V$.  Hence
\[
        \mathbb E M
        =
        \frac{n}{m}\cdot
        \frac{|X||Y|}{n(n-1)}=\Omega_{F,c}(n).
\]

Note that swapping two
entries of the permutation changes
$M$ by at most $2$. Since $|X|,|Y|\geq cn$, McDiarmid's inequality gives that for random
permutations,
\[
\Pr\left(|M-\mathbb E M|\ge \mathbb E M\right)
\le
\exp(-\Omega_{F,c}(n)).
\]
Therefore
\[
        \Pr\bigl(E_{\mathcal F}(X,Y)=\emptyset\bigr)
        \le
        \Pr(M=0)
         \le
        \Pr\left(|M-\mathbb E M|\ge \mathbb E M\right)
        \le \exp(-\Omega_{F,c}(n)). 
\]
\end{proof}

\section{The lower bound}\label{sec:lower}

We now prove the lower bound in \cref{thm:general}.  The construction is the
complete bipartite obstruction used for random perfect matchings in~\cite{DG23}, but the relevant local statistic is the path-cover number of the random induced subgraph of $F$.

\begin{proof}[Proof of the lower bound in \cref{thm:general}]
It suffices to show that for every $0\le x<\tau_{\pc}(F)$, there exists a graph
$H$ with minimum degree $xn$ such that w.h.p.~$H\cup\mathcal F$ is not
pancyclic, where $\mathcal F$ is a uniformly random $F$-factor.

Fix such an $x$. Since $x<\tau_{\pc}(F)$ and $\Delta_F(t)=mt-\Psi_F(t)$ is strictly increasing,
we have $\Psi_F(x)>mx$.  Moreover, $\pc(F-S)\le |V(F-S)|=m-|S|$ for every
$S\subseteq V(F)$, and therefore $\Psi_F(x)\le m(1-x)$.  Hence $x<1/2$.

Let $n$ tend to infinity through multiples of $m$.  Let $V=A\cup B$, where $|A|=\lfloor xn\rfloor$ and
$|B|=n-|A|$, and let $H:=K_{A,B}$. Then
$\delta(H)=|A|=\lfloor xn\rfloor$ since $|A|<|B|$.

Consider the subgraph $\mathcal F[B]$.  Since different copies of $F$ in the
factor are vertex-disjoint, the path-cover number is additive over the copies:
\[
\pc(\mathcal F[B])=
\sum_{F_i\in\mathcal F}\pc(F_i[B]).
\]
By \cref{lem:factor-set-concentration}, applied with $h=\pc$ and $U=B$, we have
w.h.p.
\[
\pc(\mathcal F[B])=
\left(\frac{\Psi_F(x)}{m}+o(1)\right)n.
\]
Since $\Psi_F(x)>mx$, it follows that w.h.p.~$\pc(\mathcal F[B])>|A|$.

Suppose now that $H\cup\mathcal F$ contains a Hamilton cycle $C$.  The graph
$H$ has no edges inside $B$, so all edges of the induced subgraph $C[B]$ lie in
$\mathcal F[B]$.  Deleting the vertices of $A$ from the cycle $C$ leaves at
most $|A|$ vertex-disjoint paths covering $B$, and all these paths use only
edges of $\mathcal F[B]$.  Thus $\pc(\mathcal F[B])\le |A|$, contradicting the
previous paragraph.  Therefore $H\cup\mathcal F$ is not Hamiltonian w.h.p., and
so $\alpha^*(F)\ge\tau_{\pc}(F)$.
\end{proof}

\section{The upper bound}\label{sec:upper}
The proof of the upper bound has two parts.  Short cycles are obtained by
closing deterministic paths with single random-factor edges. Long cycles are obtained by exposing the random factor through a random ordering and applying a Chv\'atal--Erd\H{o}s type Hamiltonicity result to each sufficiently long initial segment.

\subsection{Short cycles}

\begin{lemma}\label{lem:short-cycles}
Let $F$ be a fixed connected graph on $m\ge2$ vertices, and let $0<a<1$.  For
every sufficiently large integer $n\in m\mathbb N$ and every $n$-vertex graph
$H$ with $\delta(H)\ge an$, if $\Ffac$ is a uniformly random $F$-factor on
$V(H)$, then $H\cup\Ffac$ contains a cycle of every length
$3\le \ell\le an/2$ w.h.p.
\end{lemma}

\begin{proof}
Fix an edge $pq\in E(F)$.  For an integer $s$ with $2\le s\le an/2-1$, let
$\mathcal P_s$ be the set of ordered pairs $(x,y)\in V(H)^2$ such that $H$
contains an $x$--$y$ path of length exactly $s$.

We first show that $|\mathcal P_s|\ge an^2/2$.  Fix $x\in V(H)$.  Since
$s-1<an\le\delta(H)$, we may greedily choose a path
$x=x_0,x_1,\ldots,x_{s-1}$ of length $s-1$ in $H$.  Every vertex
$y\in N_H(x_{s-1})\setminus\{x_0,\ldots,x_{s-2}\}$ extends this path to an
$x$--$y$ path of length $s$, and there are at least $an-(s-1)\ge an/2$ choices
for $y$.  Summing over all $x$ proves the claim.

Let $M_s$ be the number of copies of $F$ in $\Ffac$ for which the ordered pair
of vertices corresponding to $(p,q)$ belongs to $\mathcal P_s$.  For each copy,
this ordered pair is uniformly distributed over all ordered pairs of distinct
vertices of $V(H)$, and so $\E M_s\ge (n/m)(an^2/2)/(n(n-1))=\Omega_{F,a}(n)$.
Interchanging two entries in the random ordering used to generate $\Ffac$
changes $M_s$ by at most two.  By \cref{lem:mcdiarmid},
$\Prob(M_s=0)\le \exp(-\Omega_{F,a}(n))$.  Taking a union bound over all
$s\le an/2-1$, we have $M_s>0$ simultaneously for all such $s$ w.h.p.  Whenever
$M_s>0$, some random-factor edge $xy$ has its endpoints joined by an
$x$--$y$ path of length $s$ in $H$, and these edges form a cycle of length
$s+1$.  Hence every length $3\le \ell\le an/2$ occurs w.h.p.
\end{proof}

\subsection{Long cycles}

For the long-cycle part we use the following sufficient condition for Hamiltonicity due to
Faudree et al.~\cite{FFGJM10}. The same statement can also
be derived from a result of Ota~\cite{Ota-1995}, with the appropriate
interpretation of the condition involving \(\ind(G)\).  For a graph \(G\), let
\(\kappa(G)\) denote its connectivity.

\begin{lemma}[{\cite[Proposition~1]{FFGJM10}}]\label{thm:FFGJM}
Let $k\ge 2$ be an integer. Suppose that $G$ is a graph on $n$ vertices with
$\kappa(G)\ge k$ and $\delta(G)>(n+k^2-k-1)/(k+1)$. If
$\delta(G)\ge \ind(G)+k-2$, then $G$ contains a Hamilton cycle.
\end{lemma}

The next lemma converts the above deterministic criterion into a probabilistic
Hamiltonicity statement for a dense host plus a random \(F\)-factor.  The small
exceptional set \(B\) is included, since an initial segment of a
random ordering need not have size divisible by \(m\) later.  Thus, the random
\(F\)-factor covers all but fewer than \(m\) vertices, while the host graph
still spans all \(n\) vertices.

\begin{lemma}\label{lem:truncated-hamilton}
Let \(F\) be a connected graph on \(m\ge2\) vertices, and let
\(a>\tau_{\ind}(F)\). There exists \(c=c(F,a)>0\) such that the following holds
for every sufficiently large integer \(n\). Let \(G\) be an \(n\)-vertex graph
with \(\delta(G)\ge an\), and let \(B\subseteq V(G)\) satisfy \(|B|<m\) and
\(m\mid(n-|B|)\). If \(\Ffac'\) is a uniformly random \(F\)-factor on
\(V(G)\setminus B\), then
\[
\Prob(G\cup\Ffac'\text{ is Hamiltonian})\ge 1-\exp(-cn).
\]
\end{lemma}

\begin{proof}
We may assume that $a<1$. Put
$n_B:=n-|B|$ and $V_B:=V(G)\setminus B$.  Choose $b$ such that
$\tau_{\ind}(F)<b<a$.  Since $x\mapsto mx-\Phi_F(x)$ is positive at $b$, we
have $\Phi_F(b)/m<b<a$.  Choose $\eta>0$ such that
\begin{equation}\label{eq:eta-choice}
        \frac{1}{m}\Phi_F(b)\le a-3\eta.
\end{equation}

For each $v\in V(G)$, let $C_v:=V(G)\setminus N_G(v)$ and
$C_v^B:=C_v\cap V_B$.  Since $\delta(G)\ge an$, we have
$|C_v^B|\le |C_v|\le (1-a)n$.  As $n_B=n-O_F(1)$ and $b<a$, for all
sufficiently large $n$ we have
\begin{equation}\label{eq:C-density}
1-\frac{|C_v^B|}{n_B}\ge b.
\end{equation}
For $v\in V(G)$, set $Z_v:=\sum_{F_i\in\Ffac'}\ind(F_i[C_v^B])$.  The family
$\{C_v^B:v\in V(G)\}$ has size at most $n\le 2n_B$ for large $n$.  Applying
\cref{lem:factor-set-concentration}, with $h=\ind$, to the random
\(F\)-factor on $V_B$, we obtain with probability at least
$1-\exp(-\Omega_{F,a}(n))$ that, for every $v\in V(G)$,
\[
        Z_v
        \le \frac{n_B}{m}\Phi_F\left(1-\frac{|C_v^B|}{n_B}\right)+\eta n_B
        \le \frac{n_B}{m}\Phi_F(b)+\eta n_B
        \le (a-2\eta)n.
\]
Here the second inequality uses \eqref{eq:C-density} and the monotonicity of
$\Phi_F$, and the last follows from \eqref{eq:eta-choice} and $n_B\le n$.
On this event, every independent set $I$ in $J:=G\cup\Ffac'$ has size at most
$(a-\eta)n$.  Indeed, if $I\ne\emptyset$, choose $v\in I$.  Since $I$ is
independent in $G$, we have $I\subseteq C_v$.  Moreover, $I\cap V_B$ is
independent in the disjoint union $\Ffac'[C_v^B]$, and so
\begin{equation}\label{eq:ind-bound}
        |I|\le Z_v+|B|\le (a-\eta)n
\end{equation}
for all sufficiently large $n$, since $|B|<m$ is fixed while $\eta n\to\infty$.

It remains to prove the required connectivity.  Choose an integer $k\ge2$ such
that $a>1/(k+1)$.  Fix $W\subseteq V(G)$ with $|W|\le k-1$, and let $D$ be a
component of $G-W$.  For any $u\in D$, all neighbours of $u$ in $G$ lie in
$D\cup W$, and therefore $|D|\ge an-|W|+1\ge an-k+2$.  Thus every component
$D$ of $G-W$ satisfies $|D\cap V_B|\ge an-k+2-(m-1)\ge an/2$ for large $n$, and
the number of components is bounded in terms of $a$ and $k$ only.

For any two distinct components $D_1,D_2$ of $G-W$, \cref{lem:linear-pairs-joined}
applied inside $V_B$ gives
\[
        \Prob\bigl(E_{\Ffac'}(D_1\cap V_B,D_2\cap V_B)=\emptyset\bigr)
        \le \exp(-\Omega_{F,a}(n)).
\]
A union bound over all $O(n^{k-1})$ choices of $W$ and all pairs of components
shows that, with probability at least $1-\exp(-\Omega_{F,a,k}(n))$, the graph
$J-W$ is connected for every $W\subseteq V(G)$ with $|W|\le k-1$.  Hence
\begin{equation}\label{eq:connectivity}
        \kappa(J)\ge k.
\end{equation}

Combining \eqref{eq:ind-bound} and \eqref{eq:connectivity}, and using the
choice of $k$, we have for all sufficiently large $n$ that
$\delta(J)\ge an>(n+k^2-k-1)/(k+1)$ and
$\delta(J)\ge an\ge \ind(J)+k-2$.  By \cref{thm:FFGJM}, $J$ is Hamiltonian.
The two exceptional events above have probability at most $\exp(-cn)$ after
decreasing $c=c(F,a)$ if necessary.
\end{proof}

We now use \cref{lem:truncated-hamilton} to find Hamilton cycles in suitable
initial segments of the random ordering.  This yields cycles of prescribed
lengths, while the exceptional set in \cref{lem:truncated-hamilton} handles
the case where the segment length is not divisible by \(m\).

\begin{lemma}\label{lem:long-cycles}
Let $F$ be a connected graph on $m\ge2$ vertices, and let
$a>\tau_{\ind}(F)$.  There is a constant $C=C(F,a)>0$ such that the following
holds.  For every sufficiently large integer $n\in m\mathbb N$ and every
$n$-vertex graph $H$ with $\delta(H)\ge an$, if $\Ffac$ is a uniformly random
$F$-factor on $V(H)$, then $H\cup\Ffac$ contains a cycle of every length
$C\log n\le \ell\le n$ w.h.p.
\end{lemma}

\begin{proof}
Generate $\Ffac$ from a uniformly random ordering $v_1,\ldots,v_n$ of $V(H)$.
For each $\ell\in[n]$, let $X_\ell:=\{v_1,\ldots,v_\ell\}$.  Choose
$b$ with $\tau_{\ind}(F)<b<a$.  By the hypergeometric Chernoff bound, there is
$\gamma=\gamma(a,b)>0$ such that, for every fixed vertex $v\in V(H)$ and every
$\ell$,
\[
        \Prob\bigl(|N_H(v)\cap X_\ell|<b\ell\bigr)
        \le \exp(-\gamma\ell).
\]
Taking $C$ sufficiently large and applying a union bound over all vertices $v$
and all $\ell\ge C\log n$, we obtain the event
\begin{equation}\label{eq:prefix-degree}
        \delta(H[X_\ell])\ge b\ell
        \qquad\text{for every }\ell\ge C\log n
\end{equation}
with probability $1-o(1)$.

Fix $\ell\ge C\log n$, and write $\ell=qm+r$, where $0\le r<m$.  Put
$B_\ell:=\{v_{qm+1},\ldots,v_{qm+r}\}$ and $X'_\ell:=X_\ell\setminus B_\ell$.
Conditioned on the two sets $X_\ell$ and $B_\ell$, the first $q$ complete
blocks in the random ordering form a uniformly random $F$-factor on $X'_\ell$.
For every fixed choice of $X_\ell$ and $B_\ell$ satisfying
$\delta(H[X_\ell])\ge b\ell$, \cref{lem:truncated-hamilton}, applied to the
$\ell$-vertex graph $H[X_\ell]$ with exceptional set $B_\ell$ and with $a$
replaced by $b$, gives
\[
        \Prob\bigl(H[X_\ell]\cup\Ffac[X_\ell]
        \text{ is not Hamiltonian}\mid X_\ell,B_\ell\bigr)
        \le \exp(-c_0\ell),
\]
where $c_0=c(F,b)>0$.  Indeed, the first $q$ complete blocks give the random
factor required by \cref{lem:truncated-hamilton}; the possible additional edges
inside the incomplete block can only help.

Let $\mathcal D_\ell$ be the event that $\delta(H[X_\ell])\ge b\ell$, and let
$\mathcal H_\ell$ be the event that $H[X_\ell]\cup\Ffac[X_\ell]$ is
Hamiltonian.  The conditional estimate above implies
$\Prob(\mathcal D_\ell\cap\mathcal H_\ell^c)\le \exp(-c_0\ell)$.  Increasing
$C$, if necessary, ensures that $\sum_{\ell\ge C\log n}\exp(-c_0\ell)=o(1)$.
Thus, with high probability, \eqref{eq:prefix-degree} holds and
$\mathcal H_\ell$ holds for every $\ell\ge C\log n$.

For each such $\ell$, the Hamilton cycle in $H[X_\ell]\cup\Ffac[X_\ell]$ is a
cycle of length $\ell$ in $H\cup\Ffac$.  This completes the proof.
\end{proof}

\begin{proof}[Proof of the upper bound in \cref{thm:general}]
Fix $a>\tau_{\ind}(F)$.  Let $H$ be an $n$-vertex graph with
$\delta(H)\ge an$, where $n$ is sufficiently large and divisible by $m$, and let
$\Ffac$ be a uniformly random $F$-factor on $V(H)$.  By
\cref{lem:short-cycles}, $H\cup\Ffac$ contains cycles of all lengths
$3\le \ell\le an/2$ w.h.p.  By \cref{lem:long-cycles}, it contains cycles of
all lengths $C\log n\le \ell\le n$ w.h.p.  Since $C\log n\le an/2$ for all
sufficiently large $n$, these two intervals cover every length $3,4,\ldots,n$.
Thus $H\cup\Ffac$ is pancyclic w.h.p.  As $a>\tau_{\ind}(F)$ was arbitrary,
$\alpha_{\pan}^*(F)\le\tau_{\ind}(F)$.
\end{proof}

Together with the lower bound from \cref{sec:lower} and the trivial inequality
$\alpha^*(F)\le\alpha_{\pan}^*(F)$, this completes the proof of
\cref{thm:general}.

\section{Concluding remarks}\label{sec:concluding}

It would be interesting to see if our techniques allow for the study
of Hamiltonicity in a deterministic graph perturbed by non-complete connected graphs.
\cref{thm:general} determines the exact threshold for clique factors. However, for non-complete connected graphs, the interval in \cref{thm:general} is genuinely
non-degenerate.

\begin{proposition}\label{prop:strict-gap}
Let $F$ be connected. Then $\tau_{\pc}(F)=\tau_{\ind}(F)$ if and only if $F$ is complete. In particular, if $F$ is connected and not complete, then
\[
        \tau_{\pc}(F)<\tau_{\ind}(F).
\]
\end{proposition}

\begin{proof}
If $F$ is complete, the equality follows from the calculation in the proof of \cref{thm:clique}. Conversely, suppose that $F$ is connected and not complete. Then $F$ contains an induced copy of $P_3$: take a shortest path between two non-adjacent vertices and use its first three vertices. Hence there exists $S\subseteq V(F)$ such that $F-S\cong P_3$. For this induced subgraph, $\ind(P_3)=2>1=\pc(P_3)$.
Since $\pc(G)\le\ind(G)$ for each graph $G$, every term in the expansion of $\Phi_F(x)-\Psi_F(x)$ is non-negative for $0<x<1$, and the term corresponding to this set $S$ is positive.  Thus $\Phi_F(x)>\Psi_F(x)$ for every $0<x<1$.
Let $t=\tau_{\pc}(F)$. Then $mt=\Psi_F(t)<\Phi_F(t)$, so $mt-\Phi_F(t)<0$. Since $x\mapsto mx-\Phi_F(x)$ is strictly increasing and vanishes at $\tau_{\ind}(F)$, we obtain $t<\tau_{\ind}(F)$.
\end{proof}

The gap in \cref{prop:strict-gap} reflects a real difference between the obstruction and the sufficient condition used here. The path-cover parameter is forced by the complete bipartite lower-bound construction, whereas the independence parameter is what is needed to invoke a Hamiltonicity criterion after the random factor is added. Determining the exact threshold for non-clique factors therefore remains open.

\begin{question}\label{que:nonclique}
Determine $\alpha^*(F)$ and $\alpha_{\pan}^*(F)$ for connected non-complete graphs $F$. In particular, what are the two thresholds when $F=P_3$?
\end{question}

\section*{Acknowledgments and AI disclosure}
Dingjia Mao and Feihong Yuan were partially supported by National Key Research and Development Program of China (2023YFA1010203).  Wenling Zhou was supported by the National Natural Science Foundation of China (12401457), the China Postdoctoral Science Foundation (2024M761780), the Natural Science Foundation of Shandong Province (ZR2024QA067), and the Young Talent of Lifting Engineering for Science and Technology in Shandong, China (SDAST2025QTA074).

AI was used to proofread the final draft of this paper.

\bibliographystyle{abbrv}
\bibliography{ref}
\end{document}